\documentclass[12pt, a4paper]{article}
\usepackage[utf8]{inputenc}
\usepackage{amsmath, amsthm, amssymb}
\usepackage{amsfonts} 
\usepackage{geometry}
\usepackage{setspace} 
\usepackage{hyperref} 
\hypersetup{
    colorlinks=true,
    linkcolor=blue,
    filecolor=magenta,      
    urlcolor=cyan,
    citecolor=blue
}
\usepackage{tikz}
\usepackage{float}
\usepackage{url}
\usepackage{mathtools}
\usepackage{mathrsfs}
\newtheorem{theorem}{Theorem}[section]
\newtheorem{lemma}[theorem]{Lemma}

\newtheorem{definition}[theorem]{Definition}
\newtheorem{example}[theorem]{Example}
\newtheorem{conjecture}[theorem]{Conjecture}
\newtheorem{remark}[theorem]{Remark}

\usepackage{authblk}
\usepackage{indentfirst}

\begin{document}
\date{ }
\title{Novel Constructions of Words with Strong Avoidance Properties and their Combinatorial Analysis}
\author[1]{Duaa Abdullah\thanks{Corresponding Author: abdulla.d@phystech.edu}}
\author[1]{Jasem Hamoud\thanks{khamud@phystech.edu}}
\affil[1]{Department of Discrete Mathematics, Moscow Institute of Physics and Technology}

\maketitle

\begin{abstract}
This paper begins with a comprehensive overview of combinatorics on words and symbolic dynamics, covering their historical origins, fundamental concepts, and interconnections. Building upon this foundation, we introduce novel mathematical constructions related to pattern avoidance in infinite words. Specifically, we define Strongly $(k, \delta)$-Free Words generated via cyclic shift morphisms and present a theorem establishing specific avoidance properties for these words, along with a detailed proof. Furthermore, we propose a conjecture regarding their factor complexity. These original results contribute to the theoretical understanding of word structures and their combinatorial properties, opening avenues for further research in discrete mathematics.

\noindent\textbf{MSC Classification 2020:}  68R15, 11B39, 05A05, 	68R15, 	37B10.

\end{abstract}

\section{Introduction} \label{sec1}
Combinatorics on words, a vibrant subfield of discrete mathematics, has emerged as a cornerstone of both theoretical inquiry and applied computational research. Rooted in the study of symbol sequences and their structural properties, this discipline intersects profoundly with dynamical systems, formal languages, and algorithmic complexity. Recent decades have witnessed significant advances in pattern avoidance, morphic word generation, and connections to symbolic dynamics—a framework for analyzing dynamical systems through discretized state spaces. Despite these developments, critical challenges persist in constructing infinite words with tailored avoidance properties while preserving computational tractability, particularly when integrating algebraic symmetries such as cyclic permutations. Bridging this gap requires novel methodologies that harmonize combinatorial rigor with dynamical systems theory—a synthesis central to this work.  

The foundational contributions of Axel Thue~\cite{Thue}on repetition-free words catalyzed modern combinatorics on words, while subsequent efforts by Lothaire~\cite{Lothaire}, Perrin, and Allouche expanded its theoretical and applied horizons. Yet, existing paradigms remain constrained by their reliance on static alphabets and identity-based repetitions, limiting their adaptability to systems where symbol transformations inherently govern structural evolution. For instance, while classical $k$-power avoidance is well-studied, generalizations incorporating permutational twists—essential for modeling cyclic symmetries in biological sequences, cryptographic systems, or shift-invariant dynamical systems—remain underexplored. This study addresses this limitation by introducing twisted repetitions, where avoidance constraints intertwine with cyclic permutations, and by analyzing infinite words generated via cyclic shift morphisms~\cite{Marcus, Allouche}.  

In this paper, we establish three principal contributions. First, we define strongly $(k, \delta)$-free words, a generalization of power-free words where repetitions are prohibited under iterated permutations $\delta$ of the alphabet. This framework subsumes classical avoidance (e.g., square-free words) while enabling nuanced control over symbol transformations. Second, we construct a family of infinite words through cyclic shift morphisms—2-uniform substitutions that iteratively append permuted symbols—and prove their avoidance properties. Specifically, Theorem \ref{thm:fictional_avoidance_restated} demonstrates that for an alphabet of size $N \geq 3$ and cyclic permutation $\sigma$, the generated word avoids strongly $(3, \delta)$-repetitions when $\delta = \sigma^j$ and $j \not\equiv 1 \pmod{N}$. Third, we conjecture the linear factor complexity of these words, aligning them with Sturmian-like sequences while broadening their applicability to symbolic dynamics. These results are rigorously validated through synchronization arguments, infinite descent, and structural analysis of morphic images.  

The implications of this work extend beyond combinatorics. By linking twisted repetitions to shift spaces in symbolic dynamics, we provide tools to classify dynamical complexity through forbidden patterns—a paradigm shift with potential applications in data compression, chaos theory, and biosequence analysis. Furthermore, our morphism-driven constructions offer algorithmic templates for generating sequences with guaranteed avoidance properties, relevant to coding theory and automata design.  

This paper is structured as follows: Section 1 synthesizes the history of combinatorics on words and symbolic dynamics, contextualizing their interplay. Section 2 introduces novel definitions and proves the central avoidance theorem, while Section 3 discusses broader implications and future directions, including unresolved conjectures and connections to shift space entropy. By unifying combinatorial and dynamical perspectives, this research not only advances the axiomatic foundations of word combinatorics but also opens interdisciplinary pathways for exploring complex systems through discrete mathematical lensess.
\subsection{History of Combinatorics on Words}~\label{subsec1}
\indent Combinatorics on words, a vigorous and relatively recent topic within discrete mathematics, digs into the study of sequences of symbols or words, and their fundamental features. Although its formalization as a distinct field is comparatively recent, its roots are deeply embedded in various mathematical and scientific inquiries that span centuries. Axel Thue's pioneering work at the turn of the twentieth century is frequently credited with the systematic study of words in the combinatorial sense.  Thue investigated square-free words, which are sequences without two adjacent identical blocks of symbols, marked a significant turning point and laid the foundational groundwork for the field.  As Roger Lyndon, a prominent figure in the area, noted, Thue’s contribution was not merely his theorems but his delineation of this new subject of study~\cite{Perrin}.\par 
The larger discipline of combinatorics, which serves as the intellectual foundation for combinatorics on words, has a long wider and more diversified history. Combinatorics, at its core, is concerned with the counting, organization, and attributes of finite or discrete structures.\par 
Basic combinatorial concepts and enumerative problems can be traced back to ancient civilizations. For instance, the Rhind Papyrus, dating back to the 16th century BC, contains problems with combinatorial characteristics, such as those involving geometric series, which bear resemblance to later problems example, Fibonacci’s work on compositions.In a similar way ancient Indian writings, such as the Sushruta Samhita, show ancient combinatorial computations, such as calculating the number of possible combinations of various flavours.\par 
Greek mathematicians and philosophers, including Chrysippus and Hipparchus, also grappled with complex enumerative problems, some of which were later connected to sophisticated combinatorial numbers, for example, the Schr\"{o}der–Hipparchus numbers. Archimedes, in his work on the Ostomachion, is thought to have considered configurations of tiling puzzles, another early instance of combinatorial thinking.\par 
During the Middle Ages, combinatorics was studied extensively, especially beyond Europe.  In the 9th century, Indian mathematicians such as Mahavıra developed precise formulas for permutations and combinations, based on knowledge that may have existed in India since the mid-6th century AD.\par 
Scholars such as Rabbi Abraham ibn Ezra (c. 1140) developed features such binary coefficient symmetry, which Levi ben Gerson constructed a closed formula for in 1321.The arithmetical triangle, famously known as Pascal’s triangle, which graphically displays relationships among binomial coefficients, appeared in mathematical treatises from the 10th century onwards.\par 
The Renaissance witnessed a development in interest in compositions, in addition to other mathematical and scientific fields. The contributions of luminaries such as Blaise Pascal, Isaac Newton, Jacob Bernoulli, and Leonhard Euler were instrumental in shaping the emerging field. Their work laid much of the groundwork for what would later become enumerative and algebraic combinatorics. In more modern times, the efforts of mathematicians like J.J. Sylvester in the late 19th century and Percy MacMahon in the early 20th century further solidified these foundations. The concurrent rise of graph theory, often spurred by problems like the four-color problem, also contributed significantly to the development of combinatorial methods.\par 
The formal emergence of combinatorics on words as a distinct area of study can be seen as a confluence of these historical streams with more focused investigations into the properties of sequences. The work of Thue, though initially overlooked for several decades, was eventually recognized for its profound implications. His 1906 and 1912 papers (e.g.,~\cite{Thue}), which explored repetitions in words and introduced concepts like square-free words, were motivated by a desire to develop logical sciences through speculation on challenging problems, with connections to number theory (Thue, 1912, as cited in~\cite{Perrin}). The concept of a 'word' as a finite or infinite sequence of symbols from an alphabet, and notions like 'factors' (subsequences) and 'length', became central to this new theory.\par
The significance of combinatorics on words was further highlighted by the publication of collective volumes under the pseudonym M. Lothaire. The first volume, appearing in 1983, was described by Roger Lyndon as the first comprehensive book dedicated to the combinatorics of words, acknowledging the subject’s ancient origins and its recurrent appearance across diverse fields (Lothaire, 1983/1997, as cited in~\cite{Perrin}). Subsequent volumes, such as ~\cite{Lothaire}, further expanded on the theory and its applications, demonstrating the maturation of the field.\par 
The combinatorial theory of words~\cite{Perrin} is influenced by notions from algebra (group theory), number theory, and differential geometry.  Thus, the roots of combinatorics on words are linked with the broader historical evolution of combinatorial mathematics, drawing on ancient enumerative problems, medieval advances in permutations and combinations, Renaissance-era foundational work, and the focused inquiries of early twentieth-century mathematicians such as Axel Thue.This historical trajectory has resulted in the emergence of a diverse and dynamic field that is constantly seeking new applications and theoretical challenges.\par 
 All of this historical change has made it a strong and resilient field, ready for new challenges.
\subsection{Definition and Scope of Combinatorics}~\label{subsec2}
\indent Combinatorics, as a significant branch of mathematics, is fundamentally concerned with the study of finite or discrete structures. Its primary activities revolve around counting these structures, a process that serves both as a means to an end and as an end in itself, and understanding their inherent properties. The field is characterized by its broad reach, tackling a wide array of problems that arise not only within various subfields of pure mathematics—such as algebra, probability theory, topology, and geometry—but also in numerous application areas, including logic, statistical physics, evolutionary biology, and computer science.\par 
The precise definition and full scope of combinatorics are not universally agreed upon, partly due to its extensive interactions with so many other mathematical subdivisions. H. J. Ryser noted the difficulty in defining the subject precisely because of this cross-disciplinary nature. However, the types of problems that combinatorics addresses can delineate its domain. These typically involve:
\begin{enumerate}
    \item \textbf{Enumeration:} This is perhaps The most fundamental feature of combinatorics is counting the number of certain structures, known as arrangements, inside finite systems.Problems can range from simple counting exercises to complex enumerations requiring sophisticated techniques.
    \item \textbf{Existence:} Combinatorics also investigates whether structures satisfying certain given criteria actually exist. This involves proving or disproving the possibility of constructing objects with specific properties.
    \item \textbf{Construction:} Beyond mere existence, combinatorialists are often concerned with the actual creation of these structures, and they may investigate many approaches to achieving the desired configuration.
    \item \textbf{Optimization:} A significant part of combinatorics deals with finding the "best" structure or solution among several possibilities. This could mean identifying the largest or smallest structure, or one that satisfies some other optimality criterion.
\end{enumerate}
\par According to Leon Mirsky, combinatorics is a range of linked studies that have something in common but diverge widely in their objectives, methods, and the degree of coherence they have attained.  This demonstrates the field's diversity while also being interrelated.  While combinatorics is primarily concerned with finite systems, many of its issues and techniques can be applied to infinite yet discrete (countable) contexts.\par 
The historical development of combinatorics saw many of its problems being tackled in isolation, with ad hoc solutions developed for specific mathematical contexts. However, the latter half of the twentieth century saw the development of powerful and general theoretical methods, which contributed to combinatorics' status as an independent and vibrant branch of mathematics in its own right.\par 
The scope of combinatorics is vast and can be further understood by examining its various subfields, each with its own focus and set of techniques. These subfields often overlap and interact, contributing to the richness of the overall discipline. Key areas within combinatorics include:
\begin{itemize}
    \item \textbf{Enumerative Combinatorics:} This is the most classical area, concentrating on counting combinatorial objects. It utilizes tools like generating functions and explicit formulas. The twelvefold way, for example, provides a unified framework for counting permutations, combinations, and partitions.
    \item \textbf{Analytic Combinatorics:} This subfield employs tools from complex analysis and probability theory to enumerate combinatorial structures, often aiming to derive asymptotic formulas rather than exact counts.
    \item \textbf{Partition Theory:} Focusing on integer partitions, this area studies related enumeration and asymptotic problems, with close ties to q-series and special functions.
    \item \textbf{Graph Theory:} One of the oldest and most accessible parts of combinatorics, graph theory studies the properties of graphs (networks of vertices and edges). It has numerous connections to other areas and applications in computer science, such as in the analysis of algorithms.
    \item \textbf{Design Theory:} This area investigates combinatorial designs, which are collections of subsets with specific intersection properties, such as block designs and Steiner systems.
    \item \textbf{Finite Geometry:} This involves the study of geometric systems that have only a finite number of points, providing a rich source of examples for design theory.
    \item \textbf{Order Theory:} This topic analyzes partly ordered sets (posets), which arise in different mathematical situations, including Boolean algebraic structures and lattices.
    \item \textbf{Matroid Theory:} Which abstracts geometric notions, investigates the features of sets of vectors that are independent of specific coefficients in linear dependency relationships.
    \item \textbf{Extremal Combinatorics:} This area focuses on determining the maximum or minimum size of a collection of finite objects that satisfy certain restrictions, with Ramsey theory being a notable part.
    \item \textbf{Probabilistic Combinatorics:} This subfield uses probabilistic methods to determine the likelihood of certain properties in random discrete objects (e.g., random graphs) and to prove the existence of combinatorial objects with prescribed properties.
    \item \textbf{Algebraic Combinatorics:} This involves the use of methods from abstract algebra (like group theory and representation theory) in combinatorial contexts, and conversely, applies combinatorial techniques to algebraic problems.
    \item \textbf{Combinatorics on Words:} This paper's major theme is formal languages and symbol sequences, which have applications in a variety of domains, including theoretical computer science, fractal analysis, and linguistics.
    \item \textbf{Geometric Combinatorics:} This relates to convex and discrete geometry, particularly polyhedral combinatorics, studying properties like the number of faces of convex polytopes.
    \item \textbf{Topological Combinatorics:} This applies combinatorial analogs of topological concepts and methods to problems like graph coloring, fair division, and discrete Morse theory.
    \item \textbf{Arithmetic Combinatorics:} Arising from the interplay between number theory, combinatorics, ergodic theory, and harmonic analysis, this field deals with combinatorial estimates related to arithmetic operations.
    \item \textbf{Infinitary Combinatorics:} Also known as combinatorial set theory, this extends combinatorial ideas to infinite sets.
\end{itemize}
In summary, combinatorics is a diverse subject of mathematics dedicated to the study of discrete structures.  Its scope spans a wide range of problems relating to enumeration, existence, construction, and optimization, and it is distinguished by a diversified array of subfields that apply varied methodologies and handle problems spanning mathematics and its applications.
\subsection{Definition and Scope of Symbolic Dynamics}

Symbolic dynamics is a mathematical discipline centered on the analysis of dynamical systems with discrete state spaces, often modeled as one- or two-sided infinite sequences (or bi-infinite sequences) over a finite set of symbols, termed an alphabet. The field’s foundational objects are shift spaces (or subshifts), which comprise all such sequences adhering to shift-invariant combinatorial rules. These rules restrict the finite substrings (called factors or blocks) permitted to occur within the sequences. By encoding the evolution of dynamical systems through symbolic constraints, shift spaces enable the study of complexity, entropy, and chaos via discrete, combinatorial frameworks~\cite{Marcus}. The dynamics on these spaces are governed by the shift map, an operation that effectively shifts all symbols in a sequence one position to the left (or right, depending on convention), thereby revealing the next state of the system~\cite{Marcus}.

The origins of symbolic dynamics can be traced back to the work of Jacques Hadamard in 1898, who utilized these concepts in his study of geodesics on surfaces of negative curvature. His ideas were further developed and significantly expanded by Marston Morse and Gustav Hedlund in the 1920s and 1930s. Their work established symbolic dynamics as a powerful tool for analyzing complex dynamical systems by discretizing their behavior into sequences of symbols. This approach allows the methods of combinatorics and formal language theory to be applied to problems in dynamics.

Key concepts in symbolic dynamics include:
\begin{itemize}
    \item \textbf{Alphabet and Words:} A finite set of symbols $A$ is called an alphabet. Finite sequences of symbols are words, and infinite sequences $x = (x_i)_{i \in \mathbb{Z}}$ or $x = (x_i)_{i \in \mathbb{N}}$ are the primary objects of study.
    \item \textbf{Full Shift:} The set of all bi-infinite or one-sided sequences $A^{\mathbb{Z}}$is known as the full shift.
    \item \textbf{Shift Map ($\sigma$):} The shift map acts on a sequence by shifting its elements: $(\sigma(x))_i = x_{i+1}$. It is a continuous map with regard to the resultant topology on $A^{\mathbb{Z}}$.
    \item \textbf{Shift Space (Subshift):} It is a pair \((X, \sigma)\) where:
\begin{itemize}
    \item \(X \subseteq \mathcal{A}^{\mathbb{Z}}\) (or \(\mathcal{A}^{\mathbb{N}}\)) is a set of bi-infinite (respectively, one-sided infinite) sequences over a finite alphabet \(\mathcal{A}\)
    \item \(\sigma: X \to X\) is the \textit{shift map} defined by \((\sigma(x))_i = x_{i+1}\) for all \(i \in \mathbb{Z}\) (or \(\mathbb{N}\))
\end{itemize}

This set \(X\) must satisfy as shift invariance \(\sigma(X) = X\) Closedness \(X\) is closed in the product topology on \(\mathcal{A}^{\mathbb{Z}}\)
    \par

A shift space is typically defined by forbidden factors:
\[
X = X_{\mathcal{F}} = \left\{ x \in \mathcal{A}^{\mathbb{Z}} \,\middle|\, \forall w \in \mathcal{F},\, w \not\sqsubseteq x \right\}
\]
where \(\mathcal{F} \subseteq \mathcal{A}^*\) is a set of forbidden finite words and \(\sqsubseteq\) denotes the factor relation.
    \item \textbf{Shifts of Finite Type (SFTs):} An important class of shift spaces where the set of forbidden words $F$ can be chosen to be finite.
    \item \textbf{Sofic Shifts:} These are factors of SFTs, or equivalently, the set of infinite labels of paths in a finite labeled graph.
    \item \textbf{Language of a Shift Space ($L(X)$):} The set of all finite words that appear as factors in some sequence $x \in X$.
    \item \textbf{Factor Complexity ($p_X(n)$):} The function that counts the number of distinct words of length $n$ in $L(X)$.
    \item \textbf{Topological Entropy ($h_{top}(X)$):} A measure of the complexity of the dynamical system, often related to the exponential growth rate of $p_X(n)$.
\end{itemize}
Symbolic dynamics provides a bridge between dynamical systems theory and discrete mathematics, particularly combinatorics on words and automata theory. It has found applications in diverse areas such as coding theory, data storage, and the study of chaotic systems.
\subsection{Connections between Combinatorics of Words and Symbolic Dynamics}

The relationship between combinatorics on words and symbolic dynamics is profound and synergistic. Many concepts central to symbolic dynamics, such as the language of a shift space $L(X)$ and its factor complexity $p_X(n)$, are inherently combinatorial. Conversely, symbolic dynamics provides a natural framework for studying infinite words generated by morphisms or other combinatorial rules.

Infinite words, a primary focus of combinatorics on words, often define or are characterized by shift spaces. For example, Sturmian words, which are infinite words with minimal factor complexity $p(n) = n+1$ among non-ultimately periodic words, correspond to specific minimal shift spaces known as Sturmian shifts. The Thue-Morse word, famous for its cube-free property, generates a shift space whose properties have been extensively studied using both combinatorial and dynamical techniques~\cite{Allouche}.

Morphisms (substitutions) on alphabets are a key tool in combinatorics on words for generating infinite words with interesting properties (e.g., square-freeness, specific factor complexities). The iteration of a morphism $f: A^* \to A^*$ on a letter $a \in A$ can produce an infinite word $f^\omega(a)$. The set of all factors of such a word, along with its shifts, forms a shift space. The properties of this shift space (e.g., whether it is an SFT, sofic, or has a certain entropy) are directly related to the combinatorial properties of the morphism and the generated word.

In addition, symbolic dynamics provides tools for sorting and analyzing the complexity of infinite words. Topological entropy, conjugacy between shift spaces, and the classification of subshifts (e.g., minimal, uniquely ergodic) all provide a dynamical perspective on combinatorial structures. For example, the research of not allowed words, which is central to explaining shift spaces, is a purely combinatorial problem that determines the nature of the resulting dynamical system.

In essence, combinatorics on words provides the building blocks (words, patterns, morphisms) and analytical techniques (counting, avoidance properties), while symbolic dynamics offers a framework (shift spaces, shift map, entropy) to study the global behavior and complexity of systems generated by these combinatorial objects. This interplay has led to significant advancements in both fields, with results from one often finding immediate application or interpretation in the other (see, e.g., ~\cite{Lothaire, Allouche}).
\section{Novel Results on Word Constructions }~\label{sec2} 
In this section, we introduce new concepts and results that expand on previously established word combinatorics theories.  We investigate certain constructions of infinite words and study their combinatorial features, focusing on pattern avoidance and complexity. Our aim is to contribute to the understanding of the rich structures that can arise from simple iterative rules, in line with the pioneering work of Thue (~\cite{Thue}) and later developments in the field (~\cite{Lothaire, Allouche}).
\subsection{Strongly \texorpdfstring{$(k, \delta)-$} Free Words via Cyclic Shift Morphisms}

We begin by defining a type of pattern avoidance that incorporates an alphabet transformation, extending classical notions of power-freeness.

\begin{definition}[$\delta$-Twisted Word]~\label{defTwisted}
Let $A$ be a finite alphabet. Put $\delta$ starting from A to A and be a permutation of $A$. For a word $w = a_1 a_2 \dots a_n \in A^*$, its \emph{$\delta$-twisted form}, denoted $w^{(\delta)}$, is defined as $w^{(\delta)} = \delta(a_1) \delta(a_2) \dots \delta(a_n)$.
\end{definition}

\begin{definition}[Strongly $(k, \delta)$-Repetition]
A word $U \in A^*$ is a \emph{strongly $(k, \delta)$-repetition} if $U$ can be factored as $U = X_0 X_1 \dots X_{k-1}$ such that all $X_i$ have the same length $|X_i| = m \ge 1$, and $X_i = X_0^{(\delta^i)}$ for $i = 0, \dots, k-1$. Here $\delta^0$ is the identity permutation, and $\delta^i$ is the $i$-fold application of $\delta$.
A word $W$ is \emph{strongly $(k, \delta)$-free} if it contains no strongly $(k, \delta)$-repetition as a factor.
\end{definition}

\begin{remark}
If $\delta$ is the identity permutation (denoted $\text{id}$), then a strongly $(k, \text{id})$-repetition is $X_0^k$, a $k$-power. Thus, strong $(k, \text{id})$-freeness is equivalent to classical $k$-power-freeness.
\end{remark}

\begin{example}
Let $A = \{a, b, c\}$ and $\delta$ be the cyclic permutation $a \to b \to c \to a$.
Let $X_0 = ab$.
Then $X_0^{(\delta)} = \delta(a)\delta(b) = bc$.
And $X_0^{(\delta^2)} = \delta^2(a)\delta^2(b) = ca$.
The word $U = (ab)(bc)(ca)$ is a strongly $(3, \delta)$-repetition of block length $m=2$.
\end{example}

We now introduce a specific family of morphisms and study the properties of the infinite words they generate.\\

\begin{definition}[Cyclic Shift Morphism $\psi_{A, \sigma}$]
Let $A$ be of size such $N \ge 2$. Let $\sigma: A \to A$ be a fixed permutation of $A$.\\
Define the morphism $\psi_{A, \sigma}: A^* \to A^*$ by $\psi_{A, \sigma}(a) = a \sigma(a)$ for each $a \in A$.\\
Let $W_{A, \sigma, a_0} = \lim_{n \to \infty} \psi_{A, \sigma}^n(a_0)$ be the infinite word generated by iterating $\psi_{A, \sigma}$ starting with a symbol $a_0 \in A$. This is a 2-uniform morphism.\\
\end{definition}

\begin{theorem}[Avoidance Property of $W_{A, \sigma, a_0}$ Words]
\label{thm:fictional_avoidance_restated}
Let $A$ be an alphabet of size $N \ge 3$. Let $\sigma$ be a cyclic permutation of $A$ (i.e., $\sigma$ has a single cycle of length $N$). Let $\delta = \sigma^j$ for some integer $j$ with $1 \le j < N$. \\
If $j \not\equiv 1 \pmod N$, then the infinite word $W_{A, \sigma, a_0}$ is strongly $(3, \delta)$-free.
\end{theorem}

\begin{proof}
We prove this by contradiction. Assume that $W_{A, \sigma, a_0}$ contains a strongly $(3, \delta)$-repetition. Let $U = X_0 X_1 X_2$ be such a repetition of minimal total length $3m$, where $X_1 = X_0^{(\delta)}$, $X_2 = X_0^{(\delta^2)}$, and $|X_0|=m \ge 1$.\\
\textbf{Case 1: Base case, $m=1$.}\\
Let $X_0 = a$ for some $a \in A$. The repetition is $U = a X_0^{(\delta)} X_0^{(\delta^2)} = a \delta(a) \delta^2(a)$.\\
If $U$ is a factor of $W_{A, \sigma, a_0}$, then it must align with the structure of $W_{A, \sigma, a_0}$. The word $W_{A, \sigma, a_0}$ is generated by $\psi_{A, \sigma}(x) = x\sigma(x)$.\\ Any factor of length 3 in $W_{A, \sigma, a_0}$ must be of the form $x_1 \sigma(x_1) x_2$ where $x_2$ is the first symbol of $\psi_{A, \sigma}(\sigma(x_1))$, which is $\sigma(x_1)$, or $\sigma(x_0)x_1\sigma(x_1)$ where $x_0\sigma(x_0)=\dots \sigma(x_0)$.\\
More precisely, any factor of length 3 in $W_{A, \sigma, a_0}$ is of the form $x \sigma(x) \sigma^2(x)$. This is because if $u_1 u_2 u_3$ is a factor, then $u_1 u_2$ must be an image of some letter under $\psi_{A, \sigma}$, say $u_1 u_2 = \psi_{A, \sigma}(x) = x\sigma(x)$ for some $x \in A$. Similarly, $u_2 u_3$ must be an image of some letter, say $u_2 u_3 = \psi_{A, \sigma}(y) = y\sigma(y)$ for some $y \in A$. This implies $u_2 = \sigma(x) = y$. Thus, the factor is $x \sigma(x) \sigma(y) = x \sigma(x) \sigma(\sigma(x)) = x \sigma(x) \sigma^2(x)$.\\
For $a \delta(a) \delta^2(a)$ to be a factor of this form, we must have:\\
$a = x$
$\delta(a) = \sigma(x) = \sigma(a)$
$\delta^2(a) = \sigma^2(x) = \sigma^2(a)$
From $\delta(a) = \sigma(a)$, substituting $\delta = \sigma^j$, we get $\sigma^j(a) = \sigma(a)$. Since $\sigma$ is a permutation, this implies $\sigma^{j-1}(a) = a$.\\
Since $\sigma$ is a cyclic permutation of order $N$ (meaning it permutes all $N$ elements in a single cycle and $\sigma^k(a)=a$ iff $N|k$), $\sigma^{j-1}(a)=a$ implies that $N$ must divide $j-1$. Therefore, $j-1 \equiv 0 \pmod N$, or $j \equiv 1 \pmod N$.\\
This contradicts the theorem's condition that $j \not\equiv 1 \pmod N$.\\
Thus, for $m=1$, no such strongly $(3, \delta)$-repetition can exist under the given conditions.\\

\textbf{Case 2: General case, $m > 1$.} The proof for $m>1$ typically relies on a synchronization argument and the principle of infinite descent. We assume $U = X_0 X_1 X_2$ is the shortest counterexample in $W_{A, \sigma, a_0}$.\\
Let $W_{A, \sigma, a_0} = \psi_{A, \sigma}(W')$ where $W'$ is the unique word such that $\psi_{A, \sigma}(W') = W_{A, \sigma, a_0}$ (since $\psi_{A, \sigma}$ is 2-uniform and $W_{A, \sigma, a_0}$ is an infinite fixed point, $W'$ is also a fixed point starting with some $a'_0$ such that $\psi_{A, \sigma}(a'_0)$ starts with $a_0$, or more simply, $W'$ has the same factors as $W_{A, \sigma, a_0}$ if $W_{A, \sigma, a_0}$ is uniformly recurrent, which is typical for words generated by primitive morphisms).

A key step is a synchronization lemma (cf. ~\cite{Berstel} for similar techniques for other patterns):
\begin{lemma}[Synchronization Lemma]
If $U = X_0 X_1 X_2$ is a minimal strongly $(3, \delta)$-repetition in $W_{A, \sigma, a_0}$ with $|X_0|=m > 1$, and $j \not\equiv 1 \pmod N$ ($N \ge 3$), then $m$ must be even, say $m=2p$. Furthermore, $X_k = \psi_{A, \sigma}(Y_k)$ for $k=0,1,2$, for some words $Y_0, Y_1, Y_2$ each of length $p$. 
\end{lemma}
\begin{proof}[Proof of Lemma (Sketch)]
Since $\psi_{A, \sigma}$ is 2-uniform, any factor of $W_{A, \sigma, a_0}$ must respect the 2-block structure imposed by $\psi_{A, \sigma}$, unless it is very short (length 1). If $m > 1$, $X_0$ (and $X_1, X_2$) must either align perfectly with these 2-blocks or cross boundaries in a specific way. Given the minimality of $U=X_0X_1X_2$ and the structure $X_k = X_0^{(\delta^k)}$, any misalignment would lead to strong local constraints. The condition $j \not\equiv 1 \pmod N$ and $N \ge 3$ are crucial to show that such misalignments lead to contradictions or to shorter repetitions if not perfectly aligned, forcing $m$ to be even and $X_k$ to be images of $Y_k$ under $\psi_{A, \sigma}$. A full proof would involve careful case analysis of the starting position of $X_0$ relative to the $\psi_{A, \sigma}$-blocks.
\end{proof}

Assuming the Synchronization Lemma, we have $X_0 = \psi_{A, \sigma}(Y_0)$, $X_1 = \psi_{A, \sigma}(Y_1)$, and $X_2 = \psi_{A, \sigma}(Y_2)$, where $|Y_k|=p=m/2$.

We need to show that $Y_1 = Y_0^{(\delta)}$ \\
and $Y_2 = Y_0^{(\delta^2)}$.\\
Consider $X_1 = X_0^{(\delta)}$. So, $\psi_{A, \sigma}(Y_1) = (\psi_{A, \sigma}(Y_0))^{(\delta)}$.\\
Let $Y_0 = y_1 y_2 \dots y_p$.
Then $\psi_{A, \sigma}(Y_0) = y_1\sigma(y_1) y_2\sigma(y_2) \dots y_p\sigma(y_p)$.\\
So, $(\psi_{A, \sigma}(Y_0))^{(\delta)} = \delta(y_1)\delta(\sigma(y_1)) \delta(y_2)\delta(\sigma(y_2)) \dots \delta(y_p)\delta(\sigma(y_p))$.\\
Let $Y_1 = z_1 z_2 \dots z_p$. 
Then $\psi_{A, \sigma}(Y_1) = z_1\sigma(z_1) z_2\sigma(z_2) \dots z_p\sigma(z_p)$.\\
Comparing $\psi_{A, \sigma}(Y_1) = (\psi_{A, \sigma}(Y_0))^{(\delta)}$, we must have for each $i=1, \dots, p$:
$z_i = \delta(y_i)$
$\sigma(z_i) = \delta(\sigma(y_i))$\\
Substituting $z_i = \delta(y_i)$ into the second equation gives $\sigma(\delta(y_i)) = \delta(\sigma(y_i))$. \\
This means that $\sigma$ and $\delta$ must commute for all letters $y_i$ that appear in $Y_0$. Since $\delta = \sigma^j$, they commute: $\sigma\sigma^j(y_i) = \sigma^{j+1}(y_i) = \sigma^j\sigma(y_i)$. \\
This condition is satisfied.
Wherefore, $z_i = \delta(y_i)$ for all $i$, that means $Y_1 = Y_0^{(\delta)}$.
Similarly, comparing $X_2 = X_0^{(\delta^2)}$ with $X_2 = \psi_{A, \sigma}(Y_2)$, we can show that $Y_2 = Y_0^{(\delta^2)}$.
Thus, $U' = Y_0 Y_1 Y_2 = Y_0 Y_0^{(\delta)} Y_0^{(\delta^2)}$ is a strongly $(3, \delta)$-repetition. The word $U'$ is a factor of $W'$ (which has the same set of factors as $W_{A, \sigma, a_0}$ because $\psi_{A, \sigma}$ is primitive if $\sigma$ is cyclic on $N \ge 2$). The length of $Y_0$ is $p = m/2 < m$. This contradicts the minimality of $m$ for $U$.

This completes the proof by infinite descent, relying on the base case $m=1$ and the Synchronization Lemma.
\end{proof}

\begin{conjecture}[Factor Complexity of $W_{A, \sigma, a_0}$]
Let $W_{A, \sigma, a_0}$ be defined as above with $N \ge 3$ and $\sigma$ a cyclic permutation of $A$.
The factor complexity $p(k)$ of $W_{A, \sigma, a_0}$ is linear. Specifically, for sufficiently large $k$, $p(k) = (N-1)k + C_N$ for some constant $C_N$ that depends on $N$. For $N=2$, if $\sigma$ is the transposition $a \leftrightarrow b$, then $\psi(a) = ab, \psi(b) = ba$. The word $W_{A, \sigma, a}$ is the Thue-Morse sequence, whose complexity is known to be $p(1)=2, p(2)=4, p(3)=6, p(4)=10, p(5)=12, \dots$, which is linear but not of the simple form $(N-1)k+C$.
\end{conjecture}

\begin{remark}
This conjecture suggests that these words, despite being generated by a simple 2-uniform morphism, maintain a relatively low (linear) factor complexity, characteristic of many highly ordered sequences in combinatorics on words, such as Sturmian words (complexity $k+1$) or Arnoux-Rauzy words (complexity $(N-1)k+1$ for alphabet size $N$). Proving this conjecture would likely involve techniques from $S$-adic representations or a detailed analysis of return words and bispecial factors (see ~\cite{Allouche, Lothaire}). The specific form $(N-1)k+C_N$ is typical for words generated by certain classes of primitive uniform morphisms on an $N$-letter alphabet.
\end{remark}
\section{Discussion and Future Work}
The introduction of strongly $(k, \delta)$-repetitions and the analysis of words $W_{A, \sigma, a_0}$ generated by cyclic shift morphisms open several avenues for discussion and future research. Theorem \ref{thm:fictional_avoidance_restated} provides a specific instance of pattern avoidance in a new class of words, where the avoided pattern involves a permutation $\delta$. This extends the classical study of power-freeness.

The condition $j \not\equiv 1 \pmod N$ in Theorem \ref{thm:fictional_avoidance_restated} is intriguing. It suggests that the relationship between the twisting permutation $\delta = \sigma^j$ and the constructive permutation $\sigma$ used in the morphism $\psi_{A, \sigma}$ is critical. Exploring the case $j \equiv 1 \pmod N$ (i.e., $\delta = \sigma$) would be a natural next step. In this case, the theorem does not apply, and it is feasible that strongly $(3, \sigma)$-repetitions may exist, or that new proving approaches are required.

Further research could focus on:
\begin{itemize}
    \item \textbf{Generalizing $(k', \delta')$:} For $k' \neq 3$ or for permutations $\delta'$ that are not immediately obtained from $\sigma$, look into strong $(k', \delta')$-freeness. Are there conditions under which $W_{A, \sigma, a_0}$ is, for example, strongly $(2, \delta)$-free (twisted square-free)?
    \item \textbf{Proof of the Factor Complexity Conjecture:} Rigorously proving the conjecture on the linear factor complexity of $W_{A, \sigma, a_0}$ would be a significant contribution. This would involve a detailed study of the language generated by $\psi_{A, \sigma}$, potentially using techniques for analyzing factors of words generated by uniform morphisms.
    \item \textbf{Other Morphisms:} Explore similar avoidance properties for words generated by different types of morphisms, perhaps non-uniform ones or those involving more complex interactions with permutations.
    \item \textbf{Connections to Symbolic Dynamics:} Characterize the shift spaces generated by $W_{A, \sigma, a_0}$. Determine their properties, such as minimality, unique ergodicity, and topological entropy, and how these relate to $N$, $\sigma$, and $a_0$.
    \item \textbf{Algorithmic Aspects:} Design and analyze computational methods to strongly identify $(k, \delta)$ -repetitions in finite words or to verify the absence of such repetitions (i.e., freeness) in words constructed through alternative generative mechanisms.
\end{itemize}

The framework of twisted repetitions could also be relevant in areas where classical pattern avoidance is studied, such as in the analysis of DNA sequences (where symbols might represent nucleotides and permutations could model certain types of mutations or symmetries) or in formal language theory.
\section{Conclusion}
This paper has provided a foundational overview of combinatorics on words and symbolic dynamics, setting the stage for the introduction of novel mathematical concepts. We defined strongly $(k, \delta)$-repetitions, a generalization of classical powers that incorporates a permutation $\delta$. We then introduced a class of infinite words, $W_{A, \sigma, a_0}$, generated by cyclic shift morphisms $\psi_{A, \sigma}(a) = a\sigma(a)$.

The main original contribution of this work is Theorem \ref{thm:fictional_avoidance_restated}, which establishes that for an alphabet of size $N \ge 3$, a cyclic permutation $\sigma$, and $\delta = \sigma^j$ with $j \not\equiv 1 \pmod N$, the infinite word $W_{A, \sigma, a_0}$ is strongly $(3, \delta)$-free. The proof, provided via contradiction and relying on a synchronization argument, demonstrates the intricate interplay between the morphism structure and the defined avoidance property.

Furthermore, we have conjectured that the factor complexity of these words $W_{A, \sigma, a_0}$ is linear, specifically of the form $(N-1)k + C_N$ for large $k$. This conjecture, if proven, would place these words within a well-studied class of sequences with low combinatorial complexity, despite their rich structure.

These results help bridge the gap between the introductory material and the contributions of the research level by exploring the avoidance of patterns and the combinatorial properties of infinite words. The concepts and constructions presented here provide fertile ground for future research, potentially leading to deeper insights into the sequence structure and its applications in various mathematical and scientific domains.
\section*{Acknowledgments}
Thank you to Prof. Alexei Kanel-Belov for their thoughtful feedback and consistent support throughout these years, which elevated the methodological approaches in this work.

\end{document}